\documentclass[a4paper,12pt]{article}

\usepackage{amsmath,amsthm,amsfonts,mathtools,amsthm,hyperref,dsfont, enumerate, xcolor, verbatim, graphicx, epstopdf, enumitem, cancel, subcaption, float, amssymb}

\usepackage{ulem}

\usepackage{multirow}
\usepackage{array}
\newcommand{\PreserveBackslash}[1]{\let\temp=\\#1\let\\=\temp}
\newcolumntype{C}[1]{>{\PreserveBackslash\centering}p{#1}}
\newcolumntype{R}[1]{>{\PreserveBackslash\raggedleft}p{#1}}
\newcolumntype{L}[1]{>{\PreserveBackslash\raggedright}p{#1}}
\newtheorem{thm}{Theorem}
\newtheorem{lem}[thm]{Lemma}

\newtheorem*{thm*}{Theorem}
\newtheorem{ass}{Assumption}

\theoremstyle{remark}
\newtheorem{rem}[thm]{Remark}

\theoremstyle{definition}

\newtheorem{example}{Example}

\newcommand{\E}{\mathbb{E}}
\newcommand{\D}{{\rm d}}
\newcommand{\dd}{{\partial}}

\newcommand{\bP}{{\mathbf P}}
\newcommand{\bE}{{\mathbf E}}

\newcommand{\ind}{{\mathbf 1}}

\everymath{\displaystyle}

\makeatletter

\graphicspath{ {./simulations/} }

\begin{document}

\title{Linear and uniform in time bound for the binary branching model with Moran type interactions}
\author{A. M. G. Cox\footnote{Department of Mathematical Sciences, University of Bath, Claverton Down, Bath, BA2 7AY, UK. Email: \texttt{a.m.g.cox@bath.ac.uk}} \and E. Horton\footnote{Department of Statistics, University of Warwick, Coventry, CV4 7AL. Email: \texttt{emma.horton@warwick.ac.uk}} \and D. Villemonais\footnote{Université de Strasbourg, IRMA, Strasbourg, France. Email: \texttt{denis.villemonais@unistra.fr}}\ \footnote{Institut Universitaire de France}}

\maketitle

\begin{abstract}
In this note, we recall the definition of the binary branching model with Moran type interactions (BBMMI) introduced in \cite{CHV2024}. In this interacting particle system, particles evolve, reproduce and die independently and, with a probability that may depend on the configuration of the whole system, the death of a particle may trigger the reproduction of another particle, while a branching event may trigger the death of another particle. We recall its relation to the Feynman-Kac semigroup of the underlying Markov evolution and improve on the $L^2$ distance between their normalisations proved in \cite{CHV2024}, when additional regularity is assumed on the process.
\newline

\noindent{\it Keywords} : interacting particle systems, branching processes, many-to-one, Markov processes, Brownian motion with drift, Moran model.
\newline

\noindent{\it  MSC:} 82C22, 82C80, 65C05, 60J25, 92D25, 60J80.
\end{abstract}

%\tableofcontents

\section{Introduction}
Branching processes and Moran-type models represent two distinct but complementary approaches to studying population dynamics and related phenomena. Moran-type processes, introduced by Moran~\cite{Moran1958}, are particularly suited for modeling finite populations influenced by mechanisms such as genetic drift, mutation, and natural selection, which can either enhance or diminish genetic diversity. The Moran model describes a system of \(N\) genes where, at random exponential intervals, two particles are chosen uniformly: one is removed while the other is duplicated, breaking the independence between particles. For an in-depth exploration of this model and its generalisations, we refer the reader to~\cite{Etheridge2011} and references therein. Furthermore, this resampling approach has been employed in various models of particle systems for the numerical solution of Feynman-Kac formulae~\cite{CloezCorujo2021,del2000moran,del2000branching,Villemonais2014}.

On the other hand, branching processes naturally model systems where events such as branching and killing occur independently. These processes arise in contexts such as population size dynamics~\cite{Jagers1989,Jagersothers1995,Lambert2005}, neutron transport~\cite{CoxHarrisEtAl2019}, genetic evolution~\cite{Marshall2008}, growth-fragmentation phenomena~\cite{bertoin2017markovian,bertoin2019feynman}, and cell proliferation kinetics~\cite{Yanev2010}. They are also studied for their theoretical properties~\cite{Jagers1989,Dawson1993,IkedaNagasawaEtAl1968,IkedaNagasawaEtAl1968a,IkedaNagasawaEtAl1969}, with a particular focus on their multiplicative behavior, scaling properties, and asymptotic dynamics over long time scales.

In~\cite{CHV2024}, a new model has been proposed, which encompasses both the Moran model and binary branching processes. In this article, the authors consider a particle system with (natural) branching and killing, as well as Moran type interactions. More precisely, when the system is initiated from $N$ particles, each particle evolves according to an independent copy of a given Markov process, $X$, until either a (binary) branching or killing event occurs.  Here, binary refers to the fact that the particle is replaced by exactly two independent copies of itself.  If such a branching event occurs, with a probability that may depend on the configuration of the whole system, another particle is removed from the system according to a selection mechanism. Similarly, if a killing event occurs, with a probability which may also depend on the configuration of the whole system, another particle is duplicated via a resampling mechanism. We refer to this model as the {\it \color{black} binary branching model with Moran interactions}, or BBMMI for short.

In the present paper, we will, for simplicity, only consider the so called $N_{min}-N_{max}$ model, which is the BBMMI model with a particular choice of selection and resampling mechanisms. Indeed, the $N_{min}-N_{max}$ model is a binary branching process whose population size is constrained to remain in $\{N_{min},\ldots,N_{max}\}$, where $2\leq N_{min}\leq N_{max}<+\infty$ are fixed. In order to constrain the size of the process, when the size of the population reaches $N_{max}$ (resp. $N_{min}$) and a natural branching (resp. killing) event occurs, we set the probability of selection (resp. resampling) to be $1$. As will be clear, our results extend to the more general situations under the appropriate regularity assumptions.

The main contributions of~\cite{CHV2024} are two-fold. First, an explicit relation between the average of the empirical distribution of the particle system and the average of the underlying Markov process $X$. Letting $m_T$ denote the empirical distribution of the particle system at time $T$, and $Q_T$ denote the first moment of the underlying binary branching process without selection and resampling, the authors show that for any $T \ge 0$,
\begin{equation}
\E\left[\Pi_T^{A}\,\Pi_T^{B} m_T(f)\right] = 
m_0Q_T(f),
% \bE_{m_0}\left[f(X_{t})\exp\left(\int_0^{t} b(X_{s}) - \kappa(X_{s})\,\D s\right)\mathbf{1}_{t<\tau_\partial}\right],
\label{M21-intro}
\end{equation}
where $\Pi_T^A$ and $\Pi_T^B$ are stochastic weights that compensate for the resampling and selection events that occur up to time $T$. Second, that after normalisation, we can give explicit bounds on the difference between the empirical particle system and the corresponding semigroup:
\begin{equation}
\left\|\frac{m_0Q_T(f)}{m_0 Q_T(\mathbf{1}_E)}- \frac{m_T(f)}{m_T(\mathbf{1}_E)}\right\|_2\leq C\,\exp(c\|b\|_\infty T) \frac{\|f\|_\infty}{m_0 Q_T(\mathbf 1_E)/{N_0}}\,\frac{1}{\sqrt{N_0}},
\label{L2bound}
\end{equation}	
where $m_0(\mathbf 1_E) = N_0$ and $C,c$ are positive constants.

As the reader will notice, the above bound is exponential in $T$, which is partly due to the generality of the setting considered in \cite{CHV2024}. The aim of the present note is to state and prove that, under suitable regularity condition, this bound can be chosen linear in $T$. 

We also mention that bounds for $L^2$ norms of the form given above have been studied in detail for interacting particle systems with constant size. We refer the reader to \cite{DelMoral2004, DelMoral2013, Rousset2006, Journel2024} and references therein for further details. 

The rest of the article is set out as follows. In section~\ref{sec:descr}, we recall the $N_{min}-N_{max}$ model introduced in \cite{CHV2024}, along with some useful notation that will be used throughout the rest of the article. In section~\ref{sec:main}, we give our main result that strengthens the bound~\eqref{L2bound} obtained in \cite{CHV2024}, followed by a discussion of the implications of this result. Finally, in section~\ref{sec:BMdrift} we discuss the case of Brownian motion with drift evolving in a bounded domain with $C^2$ boundary. The purpose of this section is to prove that the $L^2$ distance between the (normalised) semigroup associated with the branching Brownian motion with drift and the approximating $N_{min}-N_{max}$ model can be optimally bounded by $C/\sqrt{N_{min}}$. %Moreover, we also extend well-known results for Fleming-Viot interacting particle systems in this context to the $N_{min}-N_{max}$ particle system. 
% This result actually improves known bounds for the Fleming-Viot interacting particle systems (where $N_{min}=N_{max}$ and the branching rate is set to $0$). 

\section{Description of the model}
\label{sec:descr}
Let $(\Omega,\mathcal F,(X_t)_{t\in[0,+\infty)})$ be a continuous time progressively measurable Markov process with values in a measurable state space $E$.
We denote by $\bP_x$ its law when initiated at $x \in E$ and by $\bE_x$ the corresponding expectation operator. We also allow for the possibility that the Markov process $X$ is absorbed, or killed, in the sense that we consider a cemetery state $\partial \notin E$ such that $X_t \in \{\partial\}$ for all $t \ge \tau_\partial := \inf\{t \ge 0 : X_t \in \{\partial\}\}$. We extend, whenever necessary, any measurable function $f:E\to[0,+\infty)$ by $f\equiv 0$ on $\dd$. We call this `hard killing' to distinguish with the notion of `soft killing' introduced below.

We further introduce functions $b:E\to\mathbb R_+$ and $\kappa:E\to\mathbb R_+$, that denote the branching and (soft) killing rate of the Markov process. With this notation, we introduce the semigroup
\begin{align*}
	Q_t f(x)&=\bE_x\left[f(X_{t})\exp\left(\int_0^{t} (b(X_{s})-\kappa(X_s))\,\D s\right)\ind_{t<\tau_\partial}\right],
\end{align*}
defined for all bounded measurable functions $f:E\to\mathbb{R}$, $t\geq 0$ and $x\in E$,
 This defines a Feynman-Kac semigroup $(Q_t)_{t\geq 0}$, which is related to the binary branching model where particles move as copies of $X$ that are killed at rate $\kappa$ and branch at rate $b$ resulting in the creation of two independent copies of the original particle. The relation between $Q$ and this process is given by the well-known many-to-one formula, see for instance~\cite{HarrisRobertsEtAl2017} and references therein, and it has been extended to the BBMMI in~\cite{CHV2024}.
 %The purpose of this short article is to improve on the bound of the $L^2$ distance between renormalised versions of $Q_T$ and $m_T$. 

Before describing the interacting particle system associated with the branching process described above, we first introduce the following assumptions.

\begin{ass}
The branching rate $b$ is uniformly bounded. 
\end{ass}
\begin{ass}
For any $x\in E$ and $t\in [0,+\infty)$, $\bP_x(\tau_\dd=t)=0$ and $\inf_{x\in E}\bP_x(\tau_\dd>t)>0$. 
\end{ass}

We now recall the algorithmic description of the dynamics of the BBMMI particle system in the particular setting of the $N_{min}-N_{max}$ model with branching and killing rates $b$ and $\kappa$. The formal construction of the process is a non-trivial task, and is given in the supplementary material~\cite{BBMMI-sup} of~\cite{CHV2024}.
To this end, let $2\leq N_{min}\leq N_{max}<\infty$, and fix ${N}_0\in\{N_{min},\ldots,N_{max}\}$. We consider the particle system  $((X^i_t)_{i\in N_t})_{t\in[0,+\infty)}$, where $N_t$ is the number of particles in the system at time $t$.

\medskip\noindent\textbf{Evolution of the $N_{min}-N_{max}$ model.}\label{Alg1} 
\begin{enumerate}
	\item The particles $X^i$, $i\in\{1,\ldots, N_0\}$, evolve as independent copies of $X$, and we consider  the following times: 
	\begin{align*}
	\tau^{b,i}_1:=\inf\{t\geq 0,\ \int_0^t b(X^i_s)\,ds\geq e_1^{i, b}\},
	\end{align*}
    and
    \begin{align*}
    \tau^{\kappa,i}_1:=\inf\{t\geq 0,\ \int_0^t \kappa(X^i_s)\,ds\geq e_1^{i, \kappa}\},
    \end{align*}
    and
	\begin{align*}
	\tau^{\dd,i}_1:=\inf\{t\geq 0,\ X^{i}_t\in \dd\},
	\end{align*}
	 where $e_1^{i, b},e_1^{i, \kappa}$, $i = 1, \dots,{N}_0$ are exponential random variables with parameter $1$, and are independent of each other and everything else.
	 \item Denoting by $i_0$ the index of the (unique) particle such that $\tau^{b,i_0}_1\wedge \tau^{\kappa,i_0}_1\wedge \tau^{\partial,i_0}_1=\tau_1$, where $\tau_1=\min_{i\in \bar S_0} \tau^{b,i}_1\wedge \tau^{\kappa,i}_1\wedge \tau^{\partial,i}_1$, we have $N_t=N_0$ for all $t<\tau_1$ and
	 \begin{enumerate}
        \item if $\tau_1=\tau^{b,i_0}_1$, then a \textit{branching event} occurs;% : a $\bar N_0+1$-th particle is added to the system at the same position as $X^{0,i_0}_{\tau_0}$;
        \item if $\tau_1=\tau^{\kappa,i_0}_1$, then a \textit{soft killing event} occurs;% : the $i_0^{th}$ particle is  removed from the system;
	 	\item if $\tau_1=\tau^{\dd,i_0}_1$, then a \textit{hard killing event} occurs.% : the $i_0^{th}$ particle is  removed from the system.	 	
	 \end{enumerate}
	 \item Then a resampling or selection event may occur, depending on the following situations.
	 \begin{enumerate}
	 
	 	\item[] {\bf Killing.} If a (soft or hard) killing event occurred at the preceding step, then we say that $i_0$ is \textit{killed} at time $\tau_1$ and we consider the following further two cases. 
         \begin{itemize}
             \item If the total number of particles, $N_0$, is equal to $N_{min}$, particle $i_0$ is removed from the system and a \textit{resampling event} occurs: choose $j_0$ uniformly from $\{1,\ldots,N_0\}\setminus\{i_0\}$ and set
             \[
             X^{i_0}_{\tau_1}:=X^{j_0}_{\tau_1}.
             \]
             Observe that the number of particles in the system at time $\tau_1$ is then $N_{\tau_1}=N_{0}=N_{min}$.
             \item If the total number of particles, $N_0$, is larger or equal to $N_{min}+1$, then the particle $i_0$ is removed from the system and the particles are then enumerated arbitrarily from $1$ to $N_{\tau_1}=N_0-1$.
         \end{itemize} 
	 	\item[] {\bf Branching.} If a branching event occurred at the preceding step, then we say that $i_0$ has \textit{branched} at time $\tau_1$ and we consider the following further two cases.
         \begin{itemize}
             \item If the total number of particles, $N_0$, is equal to $N_{max}$, a new particle is added to the system at position $X^{i_0}_{\tau_1}$ and a \textit{selection event} occurs: choose $j_0$ at random uniformly from $\{1, \dots, N_0 + 1\}$ and remove particle  $j_0$ from the system 
             %\[
             %X^{N_0+1}_{\tau_1}=X^{i_0}_{\tau_1}.%\text{ and }\bar S_{\tau_1}:=\{\max \bar S_0+1\}\cup \bar S_0\setminus\{j_0\};
             %\]
             The particles are then enumerated arbitrarily from $1$ to $N_{\tau_1}=N_0$.
             %; we say that $j_0$ is \textit{removed} at time $\tau_1$ and that $\max S_0+1$ is \textit{born} at time $\tau_1$;
             \item  If the total number of particles, $N_0$, is less than or equal to $N_{max}-1$, then a new particle is added to the system at position $X^{i_0}_{\tau_1}$:
              \[
             X^{N_0+1}_{\tau_1}=X^{i_0}_{\tau_1}.
             \]
         \end{itemize}
	 \end{enumerate}
\end{enumerate}
After time $\tau_1$ the system evolves as independent copies of $X$ until the next killing/branching event, denoted by $\tau_2$, and at time $\tau_2$ it may undergo a resampling/selection event as above. Iterating, we define the sequence $\tau_0:=0<\tau_1<\tau_2<\cdots<\tau_n<\cdots$.

We will also make use of the following assumption, which ensures that the process described above is well defined at any time $t\geq 0$.

\begin{ass}
The sequence $(\tau_n)_{n\in\mathbb{N}}$ converges to $+\infty$ almost surely.   
\end{ass}

\begin{figure}[H]
    \centering
    \includegraphics[width=14cm]{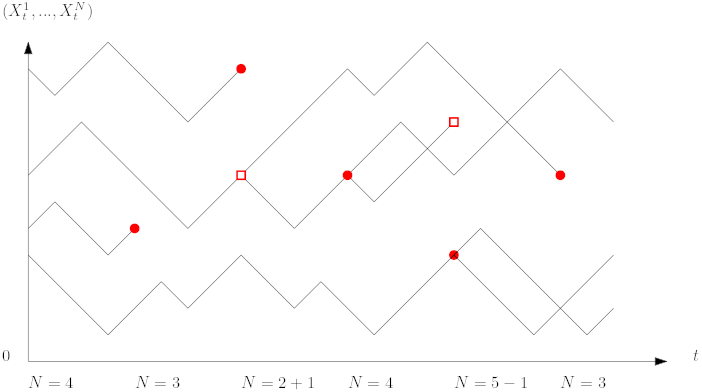}
    \caption{A schematic representation of the $N_{min}-N_{max}$ dynamic with $N_{min}=3$ and $N_{max}=4$. The process starts with $N=4$ particles at time $0$. The first event is a killing, so that the number of particles goes down to $N=3= N_{min}$. The next event is a killing, so that the number of particles goes down to $2<N_{min}$, which triggers a resampling event: one of the $2$ remaining particles (chosen uniformly at random) is duplicated, and the number of particles goes back to $N=2+1=N_{min}$. The next event is a branching, so that the number of particles goes up to $N=4= N_{max}$. The subsequent event is a branching event, so that the number of particles goes up to $5>N_{max}$, which triggers a selection event, so that one of the $5$ particles (chosen uniformly at random) is removed from the system, and the number of particles goes back to $N=4=N_{max}$. The next event  is a killing, so that the number of particles goes down to $3= N_{min}$, and so on.}
    \label{fig:NminNmax}
\end{figure}

As discussed in \cite{CHV2024}, the above model and the associated results given in \cite{CHV2024} are related to a whole suite of other models in the literature. For example, when $N_{min}=N_{max}$ and $\kappa$ bounded, we recover the standard Moran particle model (see~\cite{DelMoralMiclo2002, del2000branching,Rousset2006,CloezCorujo2021} for similar results), where the process is constrained to remain of constant size $N_0$. 

In addition, our model is reminiscent of the genetic algorithms introduced by Del Moral, see~\cite{DelMoral2004,DelMoral2013} and references therein, and also fits into the more general class of controlled branching processes introduced by Sevastyanov and Zubkov in \cite{SevastyanovZubkov1974}, where the number of reproductive individuals in each generation depends on the size of the previous generation via a control function. We refer the reader to \cite{CHV2024} for further discussion and references on these related works. 

In the rest of the article we will use the notation 
\[
  m_t := \sum_{i = 1}^{N_t}\delta_{X_t^i}, \quad t \ge 0,
\]
to denote the empirical measure associated with the $N_{min}-N_{max}$ interacting particle system. We will also use $\hat m_t$ to denote the normalised empirical measure:
\[
  \hat m_t := \frac{1}{N_t}\sum_{i = 1}^{N_t}\delta_{X_t^i}, \quad t \ge 0.
\]
The law of the particle system will be denoted by $\mathbb P$, with corresponding expectation operator~$\mathbb E$.

\section{Main result}\label{sec:main}
In this section, we present our main result, which improves on the upper bound given in Theorem 1 of \cite{CHV2024}. For this, we set
\begin{equation}
\label{eq:def-h}
h(x)=\inf_{t\geq 1} \frac{\delta_x Q_t\ind_E}{\|Q_{t-1}\ind_E\|_\infty},\ \forall x\in E
\end{equation}
and, for all bounded measurable functions $f:E\to \mathbb R$,
\begin{equation}
\label{eq:def-alpha}
\alpha_t(f)=\sup_{x\in E} \left|\frac{\delta_x Q_t f}{\delta_x Q_t\ind_E}-\nu_t(f)\right|,\ \forall \geq 0,
\end{equation}
where $\nu_t$ is a (well chosen) probability measure over $E$. 
% We also define
% \begin{align*}
%     \hat m_t=\frac1{N_t}\sum_{i=0}^{N_t} \delta_{X^i_t},\ \forall t\geq 0.
% \end{align*}
We also recall that $Q_t$, $m_t$ and $\hat m_t$, $t \ge 0$ were defined in the previous section. We will also use the notation $\mathcal M_1(E)$ to denote the collection of probability measures on $E$.

\begin{thm}\label{thm:main-result}
Under Assumptions~1,~2 and ~3, there exists a constant\footnote{Here and throughout the paper, $C$ is a positive constant that may change from line to line} $C>0$ such that, for all $T\geq 1$ and all bounded measurable functions $f:E\to \mathbb R$, we have
\begin{align}
    \label{eq:main-result}
    \left\|\frac{\hat m_0 Q_T f}{\hat m_0 Q_T \mathbf{1}_E} - \hat m_T(f)\right\|_2
    &\leq C \sum_{s=0}^{T-1} \alpha_{T-s-1}(f)\mathbb E\left(\frac{1}{\sqrt{N_{s}}\hat m_s(h)}\right).
\end{align}
\end{thm}

\medskip

 %introduce, for all $0\leq s\leq T-1$, 
%\begin{align*}
%\alpha_{s,T}(f):=\E\left(2\wedge\frac{\|Q_{T-s-1} \mathbf 1_E\|_{\infty}\sup_{\mu\in\mathcal M_1(E)}\left|\frac{\mu Q_{T-s-1} f}{\mu Q_{T-s-1}\mathbf 1_E}\right|}{\sqrt{N_s}\hat m_sQ_{T-s}\mathbf 1_E}\right).
%\end{align*}
% and 
\begin{rem}
    After the proof we will consider several examples where $\alpha_t$ is bounded by a constant, as well as situations where it decreases exponentially fast with $t$, allowing us to make use of well-known results in the theory of quasi-stationary distribution. In these situations, the right-hand side of \eqref{eq:main-result} is typically linear in $T$ and uniformly bounded over $T\geq 0$ respectively.  

    However, there are situations where $\alpha_t$ may decrease more slowly, for instance in reducible state spaces or in time inhomogeneous settings. Note that these settings are also covered by our result since $\nu_t$ is allowed to depend on $t$.
\end{rem}

\begin{proof}
   We first prove 
   that, for all $\|f\|_{\infty}\leq 1$,
\begin{align}
   \label{eq:step1}
    \left\|\frac{\hat m_0Q_T f}{\hat m_0 Q_T\mathbf{1}_E}- \frac{\hat m_1 Q_{T-1} f}{\hat m_1 Q_{T-1}\mathbf{1}_E}\right\|_2
    \leq C\frac{\|Q_{T-1} \mathbf 1_E\|_{\infty}\sup_{\mu\in\mathcal M_1(E)}\left|\frac{\mu Q_{T-1} f}{\mu Q_{T-1}\mathbf 1_E}\right|}{\sqrt{N_0} \hat m_0Q_T\mathbf 1_E}.
\end{align}
This is obtained 
via a modification of the end of the proof of Theorem~2.6 in~\cite{CHV2024}.  Denoting by $A_t$ the total number of resampling events up to time $t$, by $B_t$ the total number of selection events up to time $t$, and setting
\begin{equation}
\Pi_t^A \coloneqq \prod_{n=1}^{A_t}\left(\frac{N_{min}-1}{N_{min}}\right), \qquad \Pi_t^B \coloneqq \prod_{i=1}^{B_t}\left(\frac{N_{max}+1}{N_{max}}\right),\label{Pi}
\end{equation}
the authors obtain therein that
\[
\E\left[\left(m_0Q_1 f - \Pi_1^A\Pi_1^B m_1 f\right)^2\right]\leq  c_1 N_0  \exp\left(c_2 \|b\|_\infty\right) \left(\sup_{t\in[0,1]} \|Q_t f\|_{\infty}\right)^2,
\]
for some constants $c_1$ and $c_2$. From there, applying this result to $f=Q_{T-1} f$, one deduces that for $T \ge 1$,
\begin{align*}
    &\left\| m_0 Q_T\mathbf{1}_E \, \frac{m_1 Q_{T-1} f}{m_1 Q_{T-1}\mathbf{1}_E} -m_0Q_T f \right\|_2 \\
    &\hspace{2cm}\leq \left\|(m_0 Q_T\mathbf{1}_E-\Pi_1^A\Pi_1^B m_1Q_{T-1}\mathbf{1}_E)\,\frac{m_1 Q_{T-1} f}{m_1Q_{T-1}\mathbf{1}_E}\right\|_2+\left\|\Pi_1^A\Pi_1^B m_1 Q_{T-1} f-m_0Q_T f\right\|_2\\
    &\hspace{2cm}\leq \sqrt{c_1 N_0}  \exp\left(c_2 \|b\|_\infty/2\right) \left(\sup_{t\in[0,1]} \|Q_{t+T-1} \mathbf 1_E\|_{\infty}\sup_{\mu\in\mathcal M_1(E)}\left|\frac{\mu Q_{T-1} f}{\mu Q_{T-1}\mathbf 1_E}\right|+\sup_{t\in[0,1]} \|Q_{t+T-1} f\|_{\infty}\right).
\end{align*} 
We conclude that
\begin{multline*}
    \left\|\frac{\hat m_0Q_T f}{\hat m_0 Q_T\mathbf{1}_E}- \frac{\hat m_1 Q_{T-1} f}{\hat m_1 Q_{T-1}\mathbf{1}_E}\right\|_2\\ 
    \leq C\frac{\sup_{t\in[0,1]} \|Q_{t+T-1} \mathbf 1_E\|_{\infty}\sup_{\mu\in\mathcal M_1(E)}\left|\frac{\mu Q_{T-1} f}{\mu Q_{T-1}\mathbf 1_E}\right|+\sup_{t\in[0,1]} \|Q_{t+T-1} f\|_{\infty}}{\sqrt{N_0} \hat m_0Q_T\mathbf 1_E}.
\end{multline*}
Now note that, for all $\mu\in\mathcal M_1(E)$,
\begin{align*}
    \mu  Q_{t+T-1} f\leq \frac{\mu Q_{t+T-1} f}{\mu Q_{t+T-1} \mathbf 1_E}  \|Q_{t+T-1} \mathbf 1_E\|_{\infty}
\end{align*}
Then, since $b$ is assumed bounded, there exists a constant $C>0$ such that
\begin{align*}
    \sup_{t\in[0,1]} \|Q_{t+T-1} \mathbf 1_E\|_{\infty}\leq C \|Q_{T-1} \mathbf 1_E\|_{\infty}.
\end{align*}
Using the last two estimates in the antepenultimate inequality, we deduce that~\eqref{eq:step1} holds true.

Then, using Minkowski's inequality, the Markov property at each time $s\in \{1,2,\ldots,T-1\}$ and the first step of the proof, we obtain
\begin{align}
    \left\|\frac{\hat m_0Q_T f}{\hat m_0 Q_T\mathbf{1}_E}- \frac{\hat m_T(f)}{\hat m_T(\mathbf{1}_E)}\right\|_2
    &\leq \sum_{s=0}^{T-1} \left\|\frac{\hat m_sQ_{T-s} f}{\hat m_s Q_{T-s}\mathbf{1}_E}- \frac{\hat m_{s+1}Q_{T-s-1}f}{\hat m_{s+1}Q_{T-s-1}\mathbf{1}_E}\right\|_2\nonumber\\
    &\leq C \sum_{s=0}^{T-1}  \mathbb E\left(2\wedge \frac{\|Q_{T-s-1} \mathbf 1_E\|_{\infty}\sup_{\mu\in\mathcal M_1(E)}\left|\frac{\mu Q_{T-s-1} f}{\mu Q_{T-s-1}\mathbf 1_E}\right|}{\sqrt{N_0}\hat m_0Q_{T-s}\mathbf 1_E}\right),\label{eq:best-bound}
\end{align}
for some constant $C>0$.  Replacing $f$ by $f-\nu_t(f)$  and using the definitions of $h$ and $\alpha  $ yields the result.
\end{proof}

\begin{rem}
When the function $h$ defined in~\eqref{eq:def-h} is bounded away from $0$, then Theorem~\ref{thm:main-result} entails that there exists a constant $C>0$ such that, for all bounded measurable function $f$ and all time $T>0$,
\begin{align}
\label{eq:linconv}
    \left\|\frac{\hat m_0 Q_T f}{\hat m_0 Q_T \mathbf{1}_E} - \hat m_T(f)\right\|_2
    &\leq C \frac{T}{\sqrt{N_{min}}} \|f\|_\infty.
\end{align}
A similar estimate appeared in 
\cite[Proposition~9.5.6]{DelMoral2013}, where the particle system evolves in discrete time with a constant number of particle, under the same assumption on $h$ (transposed to discrete time).

If, in addition, $(\alpha_t(f))_{t\in\{1,2,\ldots\}}$ is summable, then there exists a constant $C_f>0$ such that, for all $T>0$,
\begin{align}
\label{eq:unifconv}
    \left\|\frac{\hat m_0 Q_T f}{\hat m_0 Q_T \mathbf{1}_E} - \hat m_T(f)\right\|_2
    &\leq \frac{C_f}{\sqrt{N_{min}}}.
\end{align}    
\end{rem}

In the rest of this section, we provide examples that illustrate typical situations where this holds true.

%This condition can for instance be proved using Harnack inequalities (see e.g. ....), including in the context of time inhomogeneous Markov processes.
%It also holds true when the semigroup $Q$ admits a bounded right eigenfunction $\varphi:E\to\mathbb R$ bounded from below. The existence of such an eigenfunction can be obtain using spectral theory (see e.g.~\cite{Rousset2006} in a similar context) or quasi-stationary distribution theory (see e.g.~\cite{ChampagnatVillemonais2016b} [CSV2024] + Del Mo).

\begin{example}
[Uniform exponential convergence with bounded soft killing rate] 
In~\cite{ChampagnatVillemonais2016b}, it has been proved that there exists a probability measure $\nu_{QS}$ on $E$, constants $C>0$ and $\alpha>0$ such that
\begin{align*}
    \left\|\frac{\mu Q_t}{\mu Q_t \ind_E}-\nu_{QS}\right\|_{TV}\leq Ce^{-\alpha t},\ \forall t\geq 0,\ \forall \mu\in \mathcal M_1(D),
\end{align*}
if, and only if, there exist constants $c_1,c_2,t_0>0$ and $\nu\in \mathcal M_1(D)$ such the two following conditions are satisfied:
\begin{align*}
    &\text{(A1).}\qquad \frac{\delta_x Q_{t_0}}{\delta_x Q_{t_0}\ind_E}\geq c_1\, \nu,\ \forall x\in E,\\
    &\text{(A2).}\qquad \inf_{t\geq 0,\, x\in E}\frac{\nu Q_t\ind_E}{\delta_x Q_t\ind_E}\geq c_2.
\end{align*}
If, in addition, $\inf_{x\in E} \delta_x Q_{t_0}\ind_E>0$, then this implies that $h$ is uniformly bounded from below and that, for all measurable function $f$, we have $\alpha_t(f)\leq Ce^{-\alpha t}\|f\|_\infty$, so that the uniform convergence~\eqref{eq:unifconv} holds true with $C_f=C\|f\|_\infty$. 

The conditions (A1) and (A2) are known to hold true in several situations (see e.g.~\cite{DelMoralMiclo2002,DelMoral2004,DelMoral2013,ChampagnatVillemonais2016b}).  In addition, they are easily generalised to the time-inhomogeneous setting, in which case the uniform convergence result follows by taking a time dependent measure in the definition of $\alpha$.

In the particular Moran model setting (that is when $N_{min}=N_{max}$), this uniform convergence result was already known (see e.g.~\cite{Rousset2006,DelMoral2013,CloezCorujo2021}) under additional regularity conditions involving the infinitesimal generator and the \textit{carré du champs} operator associated to $(Q_t)_{t\geq 0}$. Our contribution is thus that the uniform convergence holds in a more general setting and without these regularity conditions.
\end{example}

\begin{example}
[Wasserstein distance]
Consider the case where $E$ is endowed with a bounded metric $d$. Assume that $b-\kappa$ is Lipschitz and that the following assumption holds:
\begin{itemize}
    \item[(A)] There exist constants $C, \gamma > 0$ such that, for all $x, y \in E$ and $t \ge 0$, there exists a Markovian coupling\footnote{For all $x,y \in E$, a coupling measure between $\mathbb P_x$ and $\mathbb P_y$ is a probability measure $\mathbb P_{(x,y)}$ on a probability space where $(X_t^x, X_t^y)_{t \ge 0}$ is defined, such that $(X_t^z)_{t \ge 0}$ has the same distribution as $(X_t)_{t \ge 0}$ under $\mathbb P_z$ for $z = x, y$. We say the coupling is Markovian if the coupled process $(X, Y)$ is Markovian with respect to its natural filtration.}
    \[
      \mathbb E_{(x, y)}[G_t d(X_t^x, X_t^y)] \le C {\rm e}^{\gamma t}d(x, y),
    \]
    where $G_t^x = {\rm e}^{-\int_0^t \beta(X_s^x) - \kappa(X_s^x) {\rm d}s}/ \mathbb E_x [{\rm e}^{-\int_0^t \beta(X_s^x) - \kappa(X_s^x) {\rm d}s}]$ and $X_t^x$ denotes $X_t$ under $\mathbb P_x$.
\end{itemize}
Under these assumptions, it was proved in \cite{CSV2023} that there exists a probability measure $\nu_{QS}$ such that for all Lipschitz function $f:E\to\mathbb R$, we have
\begin{align}
    \left\|\frac{\mu Q_t f}{\mu Q_t \ind_E}-\nu_{QS}(f)\right\|_{Lip}\leq Ce^{-\alpha t}W_d(\mu,\nu_{QS})\|f\|_{Lip},\ \forall t\geq 0,\ \forall \mu\in \mathcal M_1(D),
\end{align}
where, for $\mu, \nu \in \mathcal M_1(D)$, $W_d(\mu, \nu) = \sup_{\|f\|_{Lip} \le 1}|\mu(f) - \nu(f)|$.

In addition, the authors show that $h$ is lower bounded in this case.
We refer the reader to \cite{CSV2023} for example of processes satisfying this condition. 

Under these conditions, we thus obtain that, for all bounded measurable function $f$,~\eqref{eq:linconv} holds true, and that, for all Lipschitz function $f$,~\eqref{eq:unifconv} holds true with $C_f=C\,\|f\|_{Lip}$. As far as we know, this result is new, even for the simpler Moran Model (i.e. $N_{min}=N_{max}$).
\end{example}

\begin{example}
[Counter-example to linear convergence when $h=0$ on some subset]
We consider the case where $E=\{a,b\}$, $a\neq b$, $X$ is the constant continuous Markov chain ($X_t=X_0$ for all $t\geq 0$, almost surely), $\kappa=\ind_{a}+2\,\ind_{b}$, and $b=0$.
For any $N\geq 2$, define the probability measure $\mu_N=\frac1N\delta_a+\frac{N-1}{N}\delta_b$, so that
\[
\frac{\mu_N Q_t\ind_a}{\mu_N Q_t\ind_E}=\frac{\frac1N\,e^{-t}}{\frac1N\,e^{-t}+\frac{N-1}{N}\,e^{-2t}}.
\]
In particular, for $t=2\ln N$,
\[
\frac{\mu_N Q_t\ind_a}{\mu_N Q_t\ind_E}= \frac{\frac1{N^3}}{\frac1{N^3}+\frac{N-1}{N}\,\frac1{N^4}}\geq \frac{1}{1+\frac1N}.
\]
Now consider the $N_{min}-N_{max}$ model with $N_0=N_{min}=N_{max}$, the same parameters $\kappa$ and $b$, and $X^1_0=a$ and $X^i_0=b$ for all $i\geq 2$. Then, with probability $1/3$, during the first event involving $X^1$, the particle jumps to $b$ and there are $0$ particle at $a$, so that $\hat m_t(a)=0$ after this event. Since this event happens with rate $3$, we deduce that
\[
\mathbb E(\hat m_t(a))\leq e^{-3t}+(1-e^{-3t})\frac23,
\]
and hence, taking again $t=2\ln N$,
\[
\mathbb E(\hat m_t(a))\leq \frac{1}{N^6}+\left(1-\frac{1}{N^6}\right)\frac23.
\]
This shows that
\[
\mathbb E\left(\left|\frac{\mu_N Q_t\ind_a}{\mu_N Q_t\ind_E}-\hat m_t(a)\right|^2\right)\geq \frac{1}{1+\frac1N}-\frac{1}{N^6}+(1-\frac{1}{N^6})\frac23\xrightarrow[N\to+\infty]{} \frac{1}{3},
\]
and hence that~\eqref{eq:linconv} does not hold true in this setting.
\end{example}

\begin{example}
[Counter-example to uniform when no uniform convergence of semi-group]

We consider as in the previous example  the case where $E=\{a,b\}$, $a\neq b$, and $X$ is the constant continuous time Markov chain ($X_t=X_0$ for all $t\geq 0$ almost surely), but with $\kappa=\ind_{a}+\ind_{b}$, and $b=0$. In this case, for any even number $N\geq 2$, we define the probability measure $\mu_N=\frac{1}2\delta_a+\frac{1}{2}\delta_b$, so that
\[
\frac{\mu_N Q_t\ind_a}{\mu_N Q_t\ind_E}=\frac12.
\]
Now consider the $N_{min}-N_{max}$ model with $N_0=N_{min}=N_{max} = N$, the same parameters $\kappa$ and $b$, and $X^{i}_0=a$ for all $i\leq N/2$ and $X^i_0=b$ for all $i\geq N/2+1$. Since the two sets do not communicate with each other, we deduce that there exists $T_N\geq 0$ such that $\mathbb P(\hat m_{T_N}(a)=0)\geq 1/4$. We thus deduce that
\begin{align*}
    \mathbb E\left(\left|\frac{\mu_N Q_{T_N}\ind_a}{\mu_N Q_{T_N}\ind_E}-\hat m_{T_N}(a)\right|^2\right)\geq \frac{1}{4}.
\end{align*}
This shows that, even though $\inf_E h>0$,~\eqref{eq:unifconv} does not hold true.
\end{example}

\section{Brownian motion with drift on a bounded $C^2$ domain.}\label{sec:BMdrift}

In the previous section, we applied our main result to situations where $h$ was uniformly bounded from below. In this section, we consider the more challenging case of a Brownian motion with drift, killed at the boundary of a $C^2$-domain. The Fleming-Viot-type particle system with this dynamic has been studied: it is known that, in this case, the empirical distribution converges uniformly in time toward the associated Feynman-Kac semigroup (see e.g. \cite{DelmoVillemonais2018}, and \cite{JM2022} when the diffusion parameter is sufficiently small) at rate $C/N^\eta$ for some $\eta\in(0,1/2)$. The purpose of this section is two-fold: to prove that the $L^2$ distance can be (optimally) bounded by $C/\sqrt{N_{min}}$, and to extend the uniform convergence obtained in the Fleming-Viot setting to the more general framework of the $N_{min}-N_{max}$ particle system.

We consider the situation where $X$ is a solution to the SDE
\[
dX_t=dW_t+q(X_t)\,\mathrm dt,\ X_0\in D,
\]
where $D$ is a bounded domain in $\mathbb R^d$, $d\geq 2$, with $C^2$ boundary, $W$ is a standard $d$-dimensional Brownian motion, and $q:D\to \mathbb R^d$ is bounded and continuous. We assume that $b$ and $\kappa$ are bounded and that the process is killed upon hitting the boundary at time $\tau_\partial=\inf\{t\geq 0,\,X_t\notin D\}$.

\begin{thm}
Assumptions~1,~2 and~3 hold true. In addition, there exists a constant $C>0$ such that, for all bounded measurable functions $f:D\to\mathbb R$,
\begin{align}
\label{eq:thmBM}
    \left\|\frac{\hat m_0Q_T f}{\hat m_0 Q_T\mathbf{1}_E}- \hat m_T(f)\right\|_2
    &\leq  \frac{C \sqrt{N_{max}}}{N_{min}}\|f\|_\infty+\mathbb E\left(\frac{C}{\sqrt{N_{0}}\,\hat m_0(\rho_D)}\right)\|f\|_\infty,\ \forall t\geq 0,
\end{align}
where $\rho_D:D\to\mathbb R_+$ denotes the distance to the boundary of $D$.
\end{thm}

\begin{proof}
Thanks to equation (3.6) and (A1--A2) of~\cite{CC-PV2018} (see also Remark~3 therein), there exists a constant $c_0>0$ such that for all $t\geq 1$ and $x \in D$,
\begin{align*}
   \frac{\delta_x Q_t\ind_D}{\sup_{y\in D}\delta_y Q_{t-1}\ind_D} \geq c_0\rho_{D}(x),
\end{align*}
where $\rho_D$ is the distance to the boundary of $D$. This implies that, for all $t\geq 0$,
\begin{equation}
\label{eq:step1-BM}
\mathbb E\left(\frac{1}{\sqrt{N_t}\hat m_t(h)}\right)\leq  \mathbb E\left(\frac{C\sqrt{N_t}}{\sum_{i=1}^{N_t} \rho_D(X^i_t)}\right)\leq \mathbb E\left(\frac{C\sqrt{N_{max}}}{\sum_{i=1}^{N_t} \rho_D(X^i_t)}\right).
\end{equation}

It is known that the distance to the boundary for the Fleming-Viot-type system can be stochastically coupled with a system of Brownian motions with drift on $[0,a]$ for some $a>0$, reflected at $0$ and $a$ (see e.g.~\cite{Villemonais2011}). From this coupling for the Fleming-Viot particle system, we derive a new coupling for the $N_{min}-N_{max}$ model.

More precisely, let $a>0$ and $D_a=\{x\in D,\ \rho_D(x)\leq a\}$ such that the distance to the boundary is $C^2$ in $D_a$ (such a vicinity of the boundary exists since the boundary is assumed to be of regularity $C^2$). Then, when a particle $X^i$ is in $D_a$, using the fact that $\Vert\nabla \rho_D\Vert_2 = 1$, Itô's formula yields
\begin{align}\label{eq:SDEdist}
    \mathrm d\rho_D(X^i_t)=\mathrm dB^i_t+r(X^i_t)\,\mathrm dt, \ i \in \{1, \dots, N_t\},
\end{align}
for some independent Brownian motions $B^i$, $i = 1, \dots, N_t$, and some bounded continuous function $r$, up to the next branching or selection or killing or resampling event. 

We now construct a family of jumping reflected Brownian motions with drift on $[0, a]$ in such a way that the sum of the positions of these particles is always bounded above by the sum of the distance between a set of $N_{min}$  particles (in the $N_{min} - N_{max}$ process) and the boundary. The construction of such a process follows similar ideas to those presented in~\cite{Villemonais2011} and so we leave the details of the construction to the reader.
More precisely, for $i \in \{1, \dots, N_{min}\}$, between events, particles move according to the following dynamics,
\begin{align*}
    \mathrm dR^i_t=\mathrm dB^i_t-\|r\|_\infty \mathrm dt+\mathrm dL^{i,0}_t-\mathrm dL^{i,a}_t, R^i_0=0,\ i \in\{1,\ldots,N_{min}\},
\end{align*}
where $L^x$ denotes the local time of $R^{i}$ at $x \in \{0, a\}$, where, for each index $i \in \{1, \dots, N_{min}\}$, $B^i$ is the same Brownian motion as in \eqref{eq:SDEdist}, and where the processes jump to $0$ at rate $\|b\|_\infty+\|\kappa\|_\infty$ independently from each other. In addition, the  $R^i$ are required to jump according to the following rules.
%corresponds to the same
% .  We define this process so that the sum of the $R^i$ remain below the 
% More precisely, in between events, the process with index $i$  evolves using the same Brownian motion $B^i$ as the distance to the boundary of its associated particle (this is identical to~[V2011] and so we do not detail the construction). 

\begin{itemize}
    \item When a particle $X_t^i$, $i = 1, \dots, N_t$, branches and does not trigger a selection event, we do nothing (so the set of reflected brownian motion does not branch).
    \item When a particle branches and triggers a selection event, 
    \begin{itemize}
        \item if the selection event removes a particle associated to a reflected Brownian motion $R^i$, $i = 1, \dots, N_{min}$, this Brownian motion jumps to $0$ and is associated to the particle newly created (at the branching event);
        \item if the selection event removes a particle which is not associated to a reflected Brownian motion, we do nothing.
    \end{itemize}
    \item When a particle is killed and there is no resampling,
    \begin{itemize}
        \item if the particle is not associated to a reflected Brownian motion, we do nothing;
        \item if the particle is associated to a reflected Brownian motion, then the Brownian motion jumps to $0$ and is associated to a new particle, not already associated to a Brownian motion.
    \end{itemize}
    \item When a particle is killed and triggers a resampling event,
    \begin{itemize}
        \item if the particle is not associated to a reflected Brownian motion, we do nothing;
        \item if the particle is associated to a reflected Brownian motion, then the reflected Brownian motion jumps to $0$ (this is only relevant for soft killing, since in the case of hard killing the associated reflected Brownian motion is already at $0$) and is now associated to the newly created particle.
    \end{itemize}
\end{itemize}

The point of this coupling is that it has the following properties:
\begin{align*}
    \sum_{i=1}^{N_t} \rho_D(X^i_t)\geq \sum_{i=1}^{N_{min}} R^i_t,\ \forall t\geq 0,
\end{align*}
while the $R^i$ are independent (the proof is very similar to the one developed in~\cite{Villemonais2011} and we leave the details to the reader).

We first remark that, on the event $\lim_{n\to+\infty}\tau_n<+\infty$,  the distance to the boundary of the particle system accumulates to $0$ in finite time, and hence that  the set of processes $R^i$, $i\in\{1,\ldots,N_{min}\}$ accumulates to $0$. Since this is not possible (by independence of the processes $R^i$), we deduce that Assumption~3 holds true. Assumption~2 also holds true since the hitting time of an elliptic diffusion has no atom.  Hence the hypotheses of Theorem~\ref{thm:main-result} are satisfied.

We now make use of the following lemma, proved at the end of this section.
\begin{lem}
    \label{lem:useful1}
    We have
\begin{align*}
    \mathbb E\left(\frac{1}{\sum_{i=1}^{N_{min}} R^i_1}\right)\leq \frac{C}{N_{min}}.
\end{align*}
for some constant $C>0$.
\end{lem}

Using~\eqref{eq:step1} and since, according to~\cite{CC-PV2018}, we have $\alpha_t(f)\leq Ce^{-\gamma t}$ for some $C,\alpha>0$, we deduce from Theorem~\ref{thm:main-result} that~\eqref{eq:thmBM} holds true.
\end{proof}

\begin{proof}[Proof of Lemma~\ref{lem:useful1}]
    Assume without loss of generality that $a=1$. Since the random variables $R^i_1$ have a bounded density $f_R$ with respect to the Lebesgue measure on $[0,a]$ and are independent, we have
    \begin{align*}
            \mathbb E\left(\frac{1}{\sum_{i=1}^{N_{min}} R^i_1}\right)
            &=\int_0^\infty \mathcal L(t)^{N_{min}}\,\mathrm dt,
    \end{align*}
    where $\mathcal L(t)$ is the Laplace transform of $R^i_1$, that is
    \begin{align*}
        \mathcal L(t) = \int_0^1 e^{-t x} f_R(x)\,\mathrm dx.
    \end{align*}
Now note that we have
    \begin{align*}
        \limsup_{t\to+\infty} t\,\mathcal L(t) &= \limsup_{t\to+\infty} \int_0^1 t\,e^{-t x} f_R(x)\,\mathrm dx 
        \leq \sup_{[0,1]} f_R\, \lim_{t\to+\infty} \int_0^1 t\,e^{-t x}\,\mathrm dx=\sup_{[0,1]} f_R.
    \end{align*}
    Hence there exist constants $c_\infty, t_\infty>0$ such that, for all $t\geq t_\infty$,
    \begin{align*}
        \mathcal L(t)\leq \frac{1}{1+c_\infty t}.
    \end{align*}
    In addition, the derivative of $\mathcal L(t)$ is given by
    \begin{align*}
        \mathcal L'(t) = - \int_0^1 x \,e^{-t x} f_R(x)\,\mathrm dx,
    \end{align*}
    which is bounded away from $0$ on compact intervals and hence, for any $t_0>0$, there exists a constant $c_0 > 0$ such that, for all $t\in[0,t_0]$,
    \begin{align*}
        \mathcal L(t)\leq \frac{1}{1+c_0 t}.
    \end{align*}
    We deduce that there exists $c>0$ such that, for all $t\geq 0$,
    \begin{align*}
        \mathcal L(t)\leq \frac{1}{1+c t}.
    \end{align*}
    In particular,
     \begin{align*}
        \mathbb E\left(\frac{1}{\sum_{i=1}^{N_{min}} R^i_1}\right)
        &\leq \int_0^\infty \left(\frac{1}{1+c t}\right)^{N_{min}}\,\mathrm dt=\frac{1}{c(N_{min}+1)}.
    \end{align*}
\end{proof}

\section*{Acknowledgements}
AMGC and EH were supported by EPSRC Grant MaThRad EP/W026899/1.

\bibliographystyle{abbrv}
\bibliography{biblio-denis}

\end{document}